\documentclass[12pt, a4paper]{article}
\usepackage{titlesec} 
\usepackage{amsfonts}
\usepackage{amsthm}
\usepackage{amssymb}
\usepackage{amsmath}
\usepackage{tikz-cd}
\usepackage{theoremref}
\usepackage{enumerate}
\usepackage{bbm}
\usepackage{authblk}
\usepackage{commath}
\usepackage{mathtools}
\usepackage{MnSymbol}
\usepackage{dsfont}

\makeatletter
\g@addto@macro\bfseries{\boldmath}
\makeatother

\newcommand{\A}{\mathcal{A}}
\newcommand{\E}{\mathcal{E}}
\newcommand{\R}{\mathbb{R}}
\newcommand{\C}{\mathbb{C}}
\newcommand{\Z}{\mathbb{Z}}

\newcommand{\conj}[1]{\overline{#1}}

\newcommand{\B}{\mathcal{B}}

\newcommand{\Hil}{\mathcal{H}}

\newcommand{\Dy}{\mathcal{D}}

\newcommand{\bh}{\mathcal{B}(\Hil)}
\renewcommand\Re{\operatorname{Re}}
\renewcommand\Im{\operatorname{Im}}

\renewcommand{\Dy}{\mathcal{D}}
\renewcommand{\S}{\mathcal{S}}

\newcommand{\einsH}{\mathbf{1_{\Hil}}}
\newtheorem{thm}{Theorem}[section]

\newtheorem{lemma}[thm]{Lemma}
\newtheorem{cor}[thm]{Corollary}
\newtheorem{prop}[thm]{Proposition}
\theoremstyle{definition}

\newtheorem{remark}[thm]{Remark}
\newtheorem*{problem*}{Problem}

\titleformat{\section}
{\normalfont\scshape\centering}{\thesection}{1em}{}
\titleformat{\subsection}
[runin]{\bfseries\scshape}{\thesubsection}{0.5em}{}
\titleformat{\subsubsection}
[runin]{}{\thesubsubsection}{0.5em}{}
\begin{document}
\title{\textbf{Sparse Lerner operators in infinite dimensions}}
\author{Adem Limani \& Sandra Pott.}
\date{
}
\maketitle
%
%
\begin{abstract}
\noindent
We use the principle of almost orthogonality to give a new and simple proof that a sparse Lerner operator is bounded on a matrix- or operator-weighted space $L_W^{2}(\mu)$, where $\mu$ is a doubling measure on $\R^d$ if and only if the weight $W$ satisfies the Muckenhoupt $A_2(\mu)$-condition, restricted to the sparse collection in question. Our method extends to the infinite-dimensional setting, thus allowing for applications to the multi-parameter setting. For the class of Muckenhoupt $A_2$-weights, we obtain bounds in terms of mixed $A_{2}(\mu)$-$A_{\infty}(\mu)$-conditions, which is independent of dimension and agrees with the best known bound in the finite-dimensional vectorial setting. As an application, we prove a matrix-weighted bound for the maximal Bergman projection, where we obtain a new sharper bound in terms of the B\'ekoll\'e-Bonami characteristic. Furthermore, we consider commutators of sparse Lerner operators on operator-valued weighted $L^{2}$-spaces and some applications to multi-parameters.
\end{abstract}
%
%
\section{Introduction}
Let $\mu$ be a locally finite positive Borel measure on the Euclidean space $\R^d$. A countable collection $\Dy$ of cubes in $\R^d$ is said to form a \emph{dyadic grid}, if the following properties hold:
\begin{enumerate}
\item[(i)] Every cube $Q\in \Dy$ has sidelength $2^k$, for some integer $k\in \Z$.

\item[(ii)] Each subcollection $\Dy_k \subset \Dy$ consisting of cube with sidelenghts $2^k$ form a partition of $\R^d$.

\item[(iii)]For every pair $Q, Q' \in \Dy$ we have $Q \cap Q' \in \{ \emptyset, Q, Q' \}$.

\end{enumerate}
We shall naturally refer to the elements $Q$ of a dyadic grid $\Dy$ as dyadic cubes and recall the \textit{standard dyadic grid} $\Dy(\R^{d})$ given by
\[ \left\{2^{n}\left( [0,1) + m\right): m\in \mathbb{Z}^{d}, n\in \mathbb{Z} \right\}.
\]
We say that a subcollection $\S \subset \Dy$ is \emph{sparse} (wrt $\mu$) if for any $Q \in \Dy$:
\begin{equation}\label{spcond} 
\sum_{Q'\in \textbf{Ch}_{\S}(Q)} \mu (Q') \leq \frac{1}{2}   \mu(Q) 
\end{equation}
where $\textbf{Ch}_{\S}(Q)$ denotes the set of maximal (wrt inclusion) cubes in $\S$, which are strictly contained in $Q$. Again, there is nothing particular with the constant $1/2$ and it may be replaced by any fixed $0< \delta <1$. For a sparse collection $\S$, we consider the corresponding \emph{sparse operator} $T^{\S}$ by 
\begin{equation}\label{spOps}
T^{\S}(f)(x) = \sum_{Q\in \S} 1_{Q}(x) \langle f \rangle_{\mu, Q},
\end{equation}
where $\langle f\rangle_{\mu, Q} := \frac{1}{\mu(Q)} \int_{Q} f d\mu$ denotes the $\mu$-average of $f$. Sparse operators and its many variations have recently attracted much attention, due to the breakthrough in 2013 where A. Lerner proved his sparse domination theorem, which essentially asserts that general Calder\'on-Zygmund operators can be pointwise bounded by sparse operators \cite{Ler13}. This result sparked a considerable interest in obtaining sharp bounds for various operators, using sparse operators. For instance, the sparse domination theorem provided a straightforward proof of the $A_{2}$-conjecture for general Calder\'on-Zygmund operators, which was initially solved by T. Hyt\"onen in \cite{Hyt12Sh}, using rather technical tools. 

Our purpose here is to consider sparse operators in a vectorial setting. To this end, we shall denote by $\Hil$ a separable Hilbert space equipped with the inner-product $(\cdot | \cdot )_{\Hil}$, and let $\B(\Hil)$ denote the space of bounded linear operators on $\Hil$, equipped with the usual operator norm. We say that $W:\R^d \rightarrow \B(\Hil)$ is an operator-valued weight, if for every vector $e\in \Hil$, the function

\begin{equation} \label{Scweight}
w_{e}(x):= ( \, W(x) \, e \, | \, e\, )_{\Hil}
\end{equation}
is a usual scalar weight. In fact, for the sake of ensuring well-defined Bochner integrals, we shall require that the operator-valued weights $W^{\pm 1}$ are both weakly locally $\mu$-integrable on $\R^d$. That is, 
for any $u,v\in \mathcal{H}$, the function $x \mapsto \, \left( W^{\pm 1}(x) \, u \, \lvert \, v \right)_{\mathcal{H}}$ is integrable on compacts subsets of $\R^d$ with respect to $\mu$ and satisfies for any dyadic cube $Q\subset \R^d$

\begin{equation*}\label{opweights} 
\abs{ \, \int_{Q} \,\left(W^{\pm 1}(x)  u \lvert  v \right)_{\Hil} d\mu(x) }  \leq C_{\mu, Q}  \norm{u}_{\Hil} \, \norm{v}_{\Hil}, 
\end{equation*}
where $C_{\mu, Q}>0$ is a constant possibly depending on $\mu$ and $Q$. The bounded linear operators that arise in this way will be denoted by $\int_Q, W^{\pm 1}  d\mu  $ and they are also automatically invertible, for every dyadic cube $Q\subset \R^d$. For a subcollection $\S \subseteq \Dy$, we say that an operator-valued weight $W: \R^d \to \bh$ is said to be a $\A^\S_2(\mu)$-weight, if 

\begin{equation}\label{A2cond}\left[W \right]_{\A^\S_2(\mu)} \, := \, \sup_{Q \in \S} \, \, \norm{\langle W \rangle^{1/2}_{\mu, Q} \, \langle W^{-1} \rangle^{1/2}_{\mu, Q}}^2_{\bh} < \, \infty.
\end{equation}
Note that if $\S= \Dy$, then we retain the collection of dyadic Muckenhoupt $A_{2}$-weights wrt $\mu$, denoted by $\A_2(\mu)$. Another important class of weights for our purposes, will be the collection of operator-valued dyadic $A_{\infty}$-weights. An operator-valued weight $W$ is said to belong to $\A_{\infty}(\mu)$, if for every $ e\in \Hil$ the scalar weights $w_{e}$ in \eqref{Scweight} satisfy the dyadic Fujii-Wilson $A_{\infty}(\mu)$-condition:
\begin{equation}\label{sAinfty} 
 \left[ w_e \right]_{A_{\infty}(\mu)} := \sup_{Q\in \Dy} \frac{1}{w_e(Q)} \int_{Q} M_\mu^{\Dy}(1_{Q}w_e)(x) d\mu(x) < \infty,
\end{equation}
where $M^{\Dy}_\mu(f)(x) := \sup_{Q\in \Dy} 1_{Q}(x) \langle f\rangle_{\mu, Q}$ denotes the dyadic maximal function. Due to the scale-invariance of the Fujii-Wilson condition in \eqref{sAinfty}, we conventionally set the dyadic $\A_{\infty}(\mu)$-constant to be
\[ \left[W\right]_{\A_{\infty}(\mu)} := \sup_{e\in \Hil } \, [w_{e}]_{A_{\infty}(\mu)} < \infty.
\] 
The dyadic $\A_\infty(\mu)$-condition is slightly weaker than the dyadic $\A_2(\mu)$-condition and one can show that $ \left[ W \right]_{\A_{\infty}(\mu)} \leq e \left[ W \right]_{\A_2(\mu)}$ (for instance, see \cite{Hyt17}). Given an operator-valued weight $W$, we denote by $L^2 _W := L^{2}_{W}\left(\mu , \Hil \right)$ the space of weakly locally $\mu$-integrable functions on $\R^d$, equipped with the norm

\[ \norm{f}^{2}_{L^{2}_{W}}  := \, \int_{\R^d} \| W^{1/2}f \|^{2}_{\Hil} d\mu = \int_{\R^d} \, \left(W  f \lvert  f \right)_{\Hil} d\mu   <  \infty.
\]
Denoting by $L^{\infty}_{0}(\mu)$ the space of complex-valued $\mu$-essentially bounded functions with compact support on $\R^d$, it is not difficult to show that
\[ L^{\infty}_{0}(\mu) \, \otimes \Hil   := \, \left\{\, \sum_{\text{Finite}} \, f \otimes e \, : \, e \in \Hil \, , \, f \in L^{\infty}_{0}(\mu) \right\}
\] 
forms a dense subspace of $L^{2}_{W}$, thus given a linear operator $T$ well-defined on scalar-valued functions $L^{\infty}_{0}(\mu)$, we denote the canonical $\Hil$-valued extension of $T$ by $T \otimes \mathds{1}$, defined on $L^{\infty}_{0} \, \otimes \Hil $ via

\[ \left( T \otimes \mathds{1} \right) \left( \sum_{\text{Finite}}  f \otimes e \right)  :=  \sum_{\text{Finite}} \, T(f)\otimes e, 
\]
where $\mathds{1}$ denotes the identity operator on $\bh$. If $T: L^{\infty}_{0}(\mu) \, \otimes \Hil \to L^2_W$ is bounded, then it follows by density that $T \otimes \mathds{1}$ will have a unique bounded extension to all of $L^2_W$, thus for the sake of abbreviation, we shall denote by $T \otimes \mathds{1}$ the unique continuous extension.

In the setting of matrix-weights $W$ of dimension $N>1$ and $\mu$ being the Lebesgue measure on $\R^d$, the following mixed $A_{2}$-$A_{\infty}$ bound was proved for general Calder\'on-Zygmund operators $T$ in \cite{NPTV17},
\begin{equation}\label{NPTVbdd}
 \| T \otimes \mathds{1_p} \|_{L^{2}_{W} \rightarrow L^{2}_{W}} \leq c_{d,N, T}  \left[ W \right]^{1/2}_{\A_2} \left[ W \right]^{1/2}_{\A_\infty } \left[ W^{-1} \right]^{1/2}_{\A_\infty}
\end{equation}
where $c_{d,N,T}>0$ is a constant depending on the dimensions and $T$. The authors introduced the technique of so-called convex body domination with sparse operators, extending the of the sparse domination technique in \cite{Ler13} by A. Lerner. In fact, the mixed bound in \eqref{NPTVbdd} is a consequence of the bound for sparse operators. Even in the scalar setting, the convex body domination technique gives new results, see e.g. \cite{isralowitz2020commutators}. The proof in \cite{NPTV17} of the convex body domination theorem heavily relies on the John-Ellipsoid theorem and equivalence of norms, tools which are both absent tools in the infinite dimensional setting. In fact, a consequence of our results is that a pointwise domination of Calder\'on-Zygmund operators by sparse operators is in general not possible in the infinite-dimensional setting. Indeed, it was proved in \cite{GPTV04}, \cite{GPTV01} that the Hilbert transform and the dyadic martingale transforms do not in general extend to a bounded linear operator in the operator-valued infinite dimensional setting, even if $W$ is a Muckenhoupt $A_2$-weight. These results relied, among others, on observations by F. Nazarov, S. Treil and A. Volberg in 1997, where they proved that the Carleson embedding theorem fails in the infinite dimensional setting \cite{NTV97}. The first positive result extending weighted boundedness results to an infinite-dimensional, operator-weighted setting was established by A. Aleman and O. Constantin in \cite{AleCon12}, where they proved that the the family of standard weighted Bergman projections are bounded on $L^{2}_{W}$, if and only if the operator-valued weight $W$ satisfies a standard weighted Bekoll\'e-Bonami condition. Sharp bounds for the standard weighted Bergman projection in the scalar-valued setting were proved by M. C. Reguera and the second author in \cite{PePo13}, using uniform (as opposed to pointwise) domination by certain sparse operators. More recently, a sparse domination of the Bergman projection on pseudoconvex domains in the matrix-weighted finite dimensional setting was obtained, where the authors in \cite{huo2020weighted} provided a slight improvement of the bound in \cite{AleCon12}. Our main purpose is to show that a certain family of sparse operators are bounded in the infinite-dimensional, operator-weighted setting of $L^2_W(\mu, \Hil)$, if and only if $W$ satisfies an appropriate $\mu$-adapted Muckenhoupt $A_2$-condition. In particular, we shall in our setting prove a similar bound to that of \eqref{NPTVbdd}, which makes this the best known bound to date, even in the finite dimensional setting. While the sparse domination of the Bergman projection in the infinite dimensional setting unfortunately still remains a mystery, we shall as an application of our results, provide a matrix-weighted, finite-dimensional bound of the maximal Bergman projection on $L^2_W$, which improves the bound obtained in \cite{AleCon12} and \cite{huo2020weighted}.
%
%
\section{Main results and outline}
\label{Mainsec} 
For a sparse family $\S$ of dyadic cubes on $\R^d$, we shall consider the family of \emph{sparse Lerner operators} $\{T^{\S}_{\psi,\varphi}\}_{\psi, \varphi}$ defined by 
\begin{equation}\label{Lernerops}
T^\S_{\psi,\varphi} = \sum_{Q \in \S} \frac{1}{\mu(Q)} \left( \psi_Q \otimes \varphi_Q \right)_{\mu}
\end{equation}
where $\frac{1}{\mu(Q)} \left( \psi_Q \otimes \varphi_Q \right)_{\mu}(f)(x) = \psi_Q(x) \langle \varphi_Q f \rangle_{\mu, Q}$ denotes the kernel representation and $\psi_{Q}, \varphi_Q$ are complex-valued functions supported on $Q$ with $\norm{\psi_Q}_{L^\infty(\mu)} \leq 1, \norm{\varphi_Q}_{L^\infty(\mu)} \leq 1$. We shall later see that this family of sparse Lerner operators naturally appear as convex bodies of sparse operators. We now state our first main result.

\begin{thm}\thlabel{SpL2W}

The family of sparse operators $\{T^{\S}_{\psi,\varphi}\}_{\psi, \varphi}$ defined in \eqref{Lernerops} extend to bounded linear operators on $L^2_W(\mu , \Hil)$ if and only if $W$ belongs to $\A^\S_2(\mu)$. Moreover, there exists an absolute constant $C >0$, such that
\begin{equation}\label{A2sp}
 \frac{1}{C} \left[ W \right]^{1/2}_{\A^\S_2(\mu)} \leq \sup_{\psi,\varphi} \| T^{\S}_{\psi,\varphi} \otimes \mathds{1} \|_{L^{2}_{W} \rightarrow L^{2}_{W}} \leq  C [ W ]^{3/2}_{\A^\S_2(\mu)},
\end{equation}
where the supremum is taken over all sequences $\{\psi_Q\}_{Q\in \S}, \{\varphi_Q \}_{Q \in \S}$ defined in the previous paragraph. 
\end{thm}
If we assume that the locally finite positive Borel measure $\mu$ on $\R^d$ satisfies the doubling condition:
\begin{equation}\label{Doublingcond}
K_\mu := \sup_{\substack{Q \subset \R^d \\ \text{cube} }} \frac{\mu(2Q)}{\mu(Q)} < \infty
\end{equation}
where $2Q$ denotes the dilate of a cube $Q$ in $\R^d$ by a factor $2$, then we actually obtain the following mixed-bound identical to \eqref{NPTVbdd}, for operator-valued dyadic Muckenhoupt weights.

\begin{cor} \thlabel{MixbddCor}
If $\mu$ satisfies the doubling condition in \eqref{Doublingcond} and $W$ is a dyadic $\A_2(\mu)$-weight, then we have the following improved mixed-bound:
\begin{equation}\label{A2Muck} 
\sup_{\psi, \varphi, \S} \| T^{\S}_{\psi, \varphi} \otimes \mathds{1} \|_{L^{2}_{W} \rightarrow L^{2}_{W}} \leq C_{\mu} \left[W \right]^{1/2}_{\A_2(\mu)} \left[W \right]^{1/2}_{\A_\infty(\mu)}  \left[W^{-1} \right]^{1/2}_{\A_\infty(\mu)}
\end{equation}
for some constant $C_{\mu}>0$, only depending on $\mu$. Note that we also take the supremum over all sparse collections $\S \subset \Dy$ of dyadic cubes in \eqref{spcond}.
\end{cor}

Some comments are now in order. The proof of \thref{SpL2W} relies on a stopping time argument, which allows us to decompose any sparse Lerner operator into a sum of simpler operators, so that the principle of almost orthogonality by M. Cotlar and E. Stein applies. We note that no self-improvement assumption on the weights $W$ are required, and in contrast to previous dimensional dependent proofs, our method extends to infinite dimensions. This in turn allows for certain applications to B\'ekoll\'e-Bonami weights and to multi-parameter settings, as we shall see in Section \ref{Bergsec} and Section \ref{AppSec}, respectively. The proof of \thref{MixbddCor} is similar to \thref{SpL2W}, but the improved bound hinges on a sharp reverse H\"older inequality on homogeneous type spaces (see Theorem 1.1, \cite{hytonen2012sharp}), which requires the doubling condition on the measures $\mu$. We stress the fact that even though the operator-weighted Hilbert transform is in general unbounded in an infinite-dimensional setting, regardless whether the operator-valued Muckenhoupt $A_2$-condition holds, we show here that sparse Lerner operators are bounded in this setting. This implies in particular that a pointwise domination of the Hilbert transform by sparse Lerner operators is not possible in the infinite-dimensional setting. 

Although, we have not managed to improve the Aleman-Constantin result in infinite dimensions, we shall in the finite dimensional setting $\Hil \cong \C^N$ provide a convex body domination result and apply \thref{SpL2W} to establish a matrix-weighted improved bound. We shall consider the family of Bergman projections $P^+_\gamma$ with $\gamma >-1$, defined on the upper half-plane $\C_+ := \{ z \in \C: \Im(z) >0 \}$ via
\begin{equation}\label{BergProj}
    P_\gamma (f)(z) = \int_{\C_+} \frac{f(\xi)}{(z-\conj{\xi})^{2+ \gamma} }dA_\gamma(\xi) \qquad z \in \C_+
\end{equation}
where $dA_\gamma (\xi) = \Im(\xi)^{\gamma}dA(\xi)$. These are projections onto the subspace of analytic functions in $L^2(dA_\gamma, \C)$. However, analyticity will play no role in our considerations, hence we shall also consider the family of maximal Bergman projections 
\begin{equation}\label{MBergProj}
P_\gamma ^+(f)(z) = \int_{\C_+} \frac{f(\xi)}{|z-\conj{\xi}|^{2+ \gamma} }dA_\gamma(\xi) \qquad z \in \C_+.
\end{equation}
In this setting, the relevant class of matrix-weights $W: \C_+ \to \B(\C^N)$ will be the so-called B\'ekoll\'e-Bonami weights $B_2( \gamma)$ on $\C_+$, defined by
\[
[W ]_{B_2( \gamma)} := \sup_{\substack{ J \subset \R \\ \text{interval}}} \norm{ \langle W \rangle_{Q_J, \gamma}^{1/2} \langle W^{-1} \rangle^{1/2}_{Q_J, \gamma} }_{\mathcal{B}(\C^N)} < \infty
\]
where $\langle f \rangle_{Q_J, \gamma} := \frac{1}{A_\gamma(Q_J)} \int_{Q_J} f dA_\gamma$ and $Q_J := \{ z \in \C_+ : \Re(z) \in J, \, \Im(z) \in (0, |J|] \}$ denotes the Carleson square associated to the interval $J \subset \R$. A striking result in \cite{AleCon12} says that both the operators $P^{(+)}_\gamma$ are bounded on $L^2_W(dA_\gamma, \Hil)$ if and only if the operator-valued weight $W$ belongs to $B_2(\gamma)$, moreover there exists a constant $C_\gamma>0$, only depending on $\gamma$ and independent of the dimension of $\Hil$, such that
\[
\frac{1}{C_\gamma} [W]^{1/2}_{B_2(\gamma)} \leq \norm{P^{(+)}_\gamma \otimes \mathds{1} }_{L^2_W \to L^2_W} \leq C_\gamma [W]^{5/2}_{B_2(\gamma)}.
\]
The upper bound was later improved in \cite{huo2020weighted} by reducing the exponent of $5/2$ to $2$, in the matrix-weighted setting on pseudo-convex domains. The content of our next result is to provide, in the setting of matrix-weights, an upper bound, which sharpens both of these results for the family of (maximal) Bergman projections defined above.

\begin{thm} \thlabel{BergL2W}
For every $\gamma >-1$, the (maximal) Bergman projection $P^{(+)}_\gamma$ is bounded on $L^2_W(dA_\gamma, \C^N)$ with
\[
\norm{P^{(+)}_\gamma \otimes \mathds{1} }_{L^2_W \to L^2_W} \leq C_{ \gamma , N} [W]^{3/2}_{B_2(\gamma)}.
\]
\end{thm}

As indicated, the proof of \thref{BergL2W} relies on \thref{SpL2W} and a convex body domination result (see \thref{CBDprop}), which we defer to Section \ref{Bergsec}. There we shall also mention the main obstacle for extending our result to infinite dimension. Due to the general nature of \thref{SpL2W}, it does not rely on any reverse H\"older property of the weight $W$, which makes it particularly useful in the setting of B\'ekoll\'e-Bonami weights. This was essentially the main obstruction in the previously treatments, which provided cruder bounds for the (maximal) Bergman projections. With these perspectives in mind, it should not come as a surprise if \thref{BergL2W} extends to even more general settings, with regards to the domain. 

This manuscript is organized as follows. In the preliminary section \ref{PrelSec}, we have collected the preparatory work for our main results, \thref{SpL2W}, \thref{MixbddCor} and \thref{BergL2W}. It is divided into subsections, consisting of decomposition of sparse collections into stopping times, sharp estimates for scalar weights. Section \ref{MResSec} is devoted to the proof of \thref{SpL2W} and \thref{MixbddCor}, while section \ref{Bergsec} contains the the proof of \thref{BergL2W}. In our final section, we provide some applications to boundedness results for commutators of sparse Lerner operators and their iterated versions in the multi-parameter setting. 

%
%

\section{Preliminary results and notations} 
\label{PrelSec}
%
%
\subsection{The sparseness condition} 
Here we shall briefly discuss a more conventional notion of sparseness and justify our seemingly stronger choice in \eqref{spcond}. Given a number $0<\delta <1$, we say that a subcollection $\mathcal{F} \subset \Dy$ is \emph{weakly $\delta$-sparse} wrt $\mu$, if for every $Q\in \mathcal{F}$ there exists Borel sets $E_{Q}\subset Q$ with the properties that $\mu(E_{Q}) \geq \delta \mu(Q)$, and such that the collection $\left\{E_{Q}\right\}_{Q\in \mathcal{F}}$ is pairwise disjoint. Evidently, every sparse collection is weakly $1/2$-sparse with $E_Q := Q \setminus \cup_{Q'\in \textbf{Ch}_{\S}(Q)} Q'$. Conversely, if $\mathcal{F}$ is a weakly $\delta$-sparse collection, then 
\[
\sum_{Q' \in \textbf{Ch}_{\mathcal{F}}(Q)} \mu (Q') \leq \frac{1}{\delta} \left( \sum_{Q' \in \textbf{Ch}_{\mathcal{F}}(Q) \cup \{Q\} } \mu(E_{Q'}) \right) - \mu(Q) \leq \left(\frac{1}{\delta} -1\right) \mu(Q).
\]
However, the constant $(1/\delta -1)$ may still exceed $1$, hence to remedy this, we pick an integer $m\geq 2$ with $(1/\delta - 1)/m \leq 1/2$. Adapting the techniques from Lemma 6.6 in \cite{LerNaz19}, we can decompose any weakly $\delta$-sparse collection $\mathcal{F}$ into a union of $m\geq 2$ disjoint sparse subcollections $\S_1,\dots,\S_m$ in the sense of \eqref{spcond}. Consequently, any weakly $\delta$-sparse Lerner operator can be written as a sum of $m$ sparse Lerner operators, thus our main results continue to hold for weakly sparse Lerner operators, at the cost of $\delta$-dependent constants. For our purposes and for the sake of convenience, we shall restrict our attention to sparse collections $\S$ in the sense of \eqref{spcond}.

\subsection{The principle of almost orthogonality via stopping times} 
\label{PAOsec}

In this section, we shall decompose any sparse collection of dyadic cubes into union of stopping times, by identifying collections of dyadic cubes with collections of vertices of graphs. These notions are deeply inspired by ideas from \cite{LerNaz19}, which we refer the reader to further details on these matters. With this decomposition at hand, we shall see that any sparse operator can be written as a sum of sparse operators, for which the principle of almost orthogonality by Cotlar and Stein can be utilized. 

Given a sparse collection $\S$, we view the cubes of $\S \subset \Dy$ as a set of vertices of a graph $\Gamma_{\S}$ by declaring that two distinct cubes $Q,\, Q^{'}\in \S$ are joined by a graph edge if either $Q\subset Q^{'}$ or $Q^{'}\subset Q$ and there is no intermediate cube $Q^{''}\in \S$, which lies strictly between $Q$ and $Q^{'}$. We say that any two cubes $Q,\,Q^{'}\in \S$ are connected if there is a path of graph edges between them. With this at hand, we can define $d_{\S}(Q',Q)$ to be the minimal number of graph edges from $Q'$ to $Q$ (if two cubes $Q',Q\in \S$ are not connected, we set $d_{\S}(Q',Q)=\infty$ by default). Connectedness of cubes in $\S$ induces an equivalence relation on $\Gamma_{\S}$, hence $\Gamma_{\S}$ decomposes into a collection of at most finitely many connected subgraphs. Each connected subgraph of $\Gamma_{\S}$ can be viewed as a branching tree, with the natural motions of either moving upwards to larger cubes within the subgraph or moving downwards to smaller cubes within the subgraph. Now pick exactly one vertex in each connected subgraph of $\Gamma_{\S}$ and denote the collection of the cubes corresponding to these vertices by $\mathcal{J}^{0}\subset \S$. Recursively, we may define the stopping times  

\begin{equation*}\mathcal{J}^{n+1} = \, \bigcup_{Q \in \mathcal{J}^{n}}\, \textbf{Ch}_{\S}(Q) \qquad  \qquad \mathcal{J}^{-(n+1)} = \, \bigcup_{Q\in \mathcal{J}^{-n} } \, \textbf{Pr}_{\S}(Q) \qquad n\geq 0.
\end{equation*}
\noindent
Here $\textbf{Ch}_{\S}(Q)$ denotes the maximal cubes in $\S$ which are strictly contained in $Q$ and $\textbf{Pr}_{\S}(Q)$ denotes the minimal cube in $\S$, which strictly includes $Q$. For positive integers $n$, $\mathcal{J}^{n}$ corresponds to moving down $n$ generations in all the connected subgraphs of $\S$, from every cube in $\mathcal{J}^{0}$, while for negative integers $n$, it corresponds to moving up $n$ generations in all the connected subgraphs of $\S$, from every cube in $\mathcal{J}^{0}$. In a similar way, we define the $n$-generation stopping time relative to an arbitrary cube $Q\in \S$ by

\begin{equation*} \mathcal{J}^{n}(Q):= \left\{Q'\in \S: d_{\S}(Q',Q) = n \,, \,Q' \subsetneq Q \right\} \qquad \, n\geq 0.
\end{equation*}
With these constructions at hand, we obtain a collection of families $\left\{\mathcal{J}^{n}\right\}_{n\in \mathbb{Z}}$, satisfying the following properties;

\begin{itemize}
\item[(i)] $\mathcal{J}^{n}$ is a disjoint collection of cubes in $\S$, for all $n\in \mathbb{Z}$. 
\item[(ii)] $\bigcup_{n\in \mathbb{Z}} \, \mathcal{J}^{n} \, =\, \S$.
\item[(iii)] For every $Q\in \S$, the collection $\left\{\mathcal{J}^{n}(Q) \right\}_{n=0}^{\infty}$ forms a \emph{decaying stopping time family}. That is, for every $Q\in \S$, we have

\begin{equation}\label{DSTn}\sum_{Q'\in \mathcal{J}^{n}(Q)} \, \mu(Q') \, \leq \, 2^{-n} \, \mu(Q) \qquad \, n\geq 0.
\end{equation} 
\end{itemize}
\noindent
These properties are all immediate consequences of the constructions of $\mathcal{J}^{n}$, while the third property incorporates an iteration of the sparseness condition of $\S$. Decomposing the sparse family $\S$ in this way, we may express any sparse operator as

\begin{equation*}\label{trunc} T^{\S}_{\psi, \varphi} = \ \sum_{n\in \mathbb{Z}} \, T_{n},
\end{equation*} 
where each term consist of sums of averaging operators restricted to a disjoint family $\mathcal{J}^{n}$, given by

\begin{equation}\label{aoterm}T_{n}:= \sum_{Q\in \mathcal{J}^{n}} \, \frac{1}{\mu(Q)} \, \left( \psi_{Q} \otimes \varphi_{Q} \right)_\mu \qquad \, n\in \mathbb{Z}.
\end{equation}
It turns out that the decaying stopping time property in ($\ref{DSTn}$) makes the family of operators $\left\{T_{n}\right\}_{n\in\mathbb{Z}}$ in $(\ref{aoterm})$ "almost orthogonal". In order to make the notion of almost orthogonality more precise, we will need the following tailor made version of the Cotlar-Stein lemma.

\begin{lemma}[Cotlar-Stein type lemma]\thlabel{gencotlar} 

Let $\left\{T_{n}\right\}_{n\in \mathbb{Z}}$ be a sequence of bounded linear operators on a Hilbert space $\mathcal{H}$ and suppose there are sequences of positive real numbers $\left\{\alpha(n)\right\}_{n\in \mathbb{Z}}$, $\left\{\beta(n)\right\}_{n\in \mathbb{Z}}$, with the properties 

\begin{equation}\label{Cotest}  
\begin{split}
\norm{T_{n}^{*} \, T_{m}}_{\mathcal{B}(\mathcal{H})} \, \leq \, \alpha\left(\, n-m\, \right),  \\
\norm{T_{n} \, T_{m}^{*}}_{\mathcal{B}(\mathcal{H})} \, \leq \, \beta \left(\, n-m\, \right).
\end{split}
\end{equation}
for all $m,n\in \mathbb{Z}$. Furthermore, assume that 

\[ A:=\sum_{n\in \mathbb{Z}} \, \sqrt{\alpha(n)}< \infty \qquad , \qquad B:=  \sum_{n\in \mathbb{Z}}\, \sqrt{\beta(n)} < \infty.
\] 
\noindent
Then the operator $\sum_{n} T_n$ converges unconditionally and enjoys the bound
\begin{equation*}\label{cot1}\norm{\,\sum_{n\in \mathbb{Z}} \, T_{n} \, }_{\mathcal{B}(\mathcal{H})} \, \leq \, 2 \, \sqrt{AB}.
\end{equation*}

\end{lemma}
\noindent 
The standard proof of Lemma 8.5.1, in \cite{Gra14}, can be easily be adapted to prove this version of the lemma, thus we omit the proof. 

\subsection{The $\A_2(\mu)$-condition}

Now in order to satisfy the hypothesis of the lemma \thref{gencotlar}, we necessarily need to establish boundedness of the $T_{n}$'s, uniformly in $n\in \mathbb{Z}$, which accounts for the diagonal case $m=n$ in the hypothesis \eqref{Cotest}. Since $\mathcal{J}^{n}$ is a disjoint collection, it suffices to find a uniform bound for the individual terms of $T_{n}$. This task is captured by the following lemma.

\begin{lemma}\thlabel{diag}
For any pair of complex-valued functions $\psi_Q, \varphi_Q$ supported on a cube $Q \subset \R^d$ and satisfying $\norm{\psi_Q}_{L^\infty(\mu)} \leq 1 , \norm{\varphi_Q}_{L^\infty(\mu)} \leq 1$, we have  

\begin{equation*}
\norm{\frac{1}{\mu(Q)}\, \left(\psi_Q \otimes \varphi_Q\right)_\mu }_{L^{2}_{W}\rightarrow L^{2}_{W}} \leq  \norm{\langle W \rangle_{\mu, Q}^{1/2} \, \langle W^{-1} \rangle_{\mu, Q}^{1/2}}_{\B(\Hil)}.
\end{equation*} 

\end{lemma}

\begin{proof} Fix an arbitrary $f\in L^{\infty}_{0}(\mu)\otimes \Hil$, so that $\norm{\psi_{Q} \, \langle \varphi_Q f \rangle_{\mu, Q} }_{L^{2}_{W}}< \infty$, and note that by Cauchy-Schwartz inequality, we can write
\begin{equation} \label{I} 
\begin{split}\norm{\psi_{Q} \, \langle \varphi_Q f \rangle_{\mu, Q} }^{2}_{L^{2}_{W}}= \int_{Q}\left( W^{-1/2} \langle |\psi_Q|^2 W \rangle_{\mu, Q} \langle \varphi_Q f \rangle_{\mu, Q}  \lvert \,W^{1/2} (\varphi_Q f )\right)_{\Hil} d\mu 
\\
\leq
\left(\int_{Q} \norm{W^{-1/2} \langle |\psi_Q|^2 W \rangle_{\mu, Q} \langle \varphi_Q f \rangle_{\mu, Q} }^{2}_{\Hil}d\mu \right)^{1/2} \norm{\varphi_{Q}f}_{L^{2}_{W}}.
\end{split}
\end{equation}
\noindent
We now estimate the integral on the right hand side of (\ref{I}), according to

\begin{multline*}
 \int_{Q} \norm{W^{-1/2} \langle |\psi_Q|^2 W \rangle_{\mu, Q} \langle \varphi_Q f \rangle_{\mu, Q} }^{2}_{\Hil}d\mu  =  \\
  \mu(Q) \norm{\langle W^{-1} \rangle_{\mu,Q}^{1/2} \langle |\psi_Q|^2 W \rangle_{\mu, Q}^{1/2} \left( \langle |\psi_Q|^2 W \rangle_{\mu, Q}^{1/2} \langle \varphi_Q f \rangle_{\mu, Q}\right)}^{2}_{\Hil} \leq  
 \\
 \norm{\langle W^{-1} \rangle_{\mu,Q}^{1/2} \langle |\psi_Q|^2 W \rangle_{\mu, Q}^{1/2} }^2_{\B(\Hil)} \mu(Q) \norm{\langle |\psi_Q|^2 W\rangle_{\mu, Q}^{1/2} \langle \varphi_Q f \rangle_{\mu, Q}}^{2}_{\Hil} 
= \\
\norm{\langle W^{-1} \rangle_{\mu,Q}^{1/2} \langle |\psi_Q|^2 W \rangle_{\mu, Q}^{1/2} }^2_{\B(\Hil)} \norm{\psi_{Q} \langle \varphi_Q f \rangle_{\mu, Q}}_{L^{2}_{W}}^{2}.
\end{multline*}
Going back to the expression in $(\ref{I})$ and cancelling the common factors, we obtain

\begin{equation*}\norm{\psi_Q \, \langle \varphi_Q f \rangle_{\mu, Q} }_{L^{2}_{W}}\, \leq \norm{\langle W^{-1} \rangle_{\mu, Q}^{1/2} \, \langle |\psi_Q|^2 W \rangle^{1/2}_{\mu, Q}}_{\B (\Hil)}\, \norm{\varphi_Q\, f}_{L^{2}_{W}}.
\end{equation*}
Note that by the $C^*$-identity, we can write
\begin{multline*} \label{phipsi}
\norm{\langle  W^{-1} \rangle^{1/2}_{\mu,Q} \langle |\psi_Q|^2 W \rangle^{1/2}_{\mu, Q} }^{2}_{\bh} =  
\norm{ \langle W^{-1} \rangle^{1/2}_{\mu,Q}\langle |\psi_Q|^2 W \rangle_{\mu,Q} \langle W^{-1} \rangle^{1/2}_{\mu,Q} }_{\bh} = \\
\sup_{\norm{e}_{\Hil} =1} \left( \langle |\psi_Q|^2 W \rangle_{\mu, Q} \langle W^{-1} \rangle^{1/2}_{\mu, Q} e \lvert \langle  W^{-1} \rangle^{1/2}_{\mu, Q} e \right)_{\Hil}.
\end{multline*}
Expanding $\langle |\psi_Q|^2 W \rangle_Q$ and estimating the positive function $\norm{\psi_Q}_{L^\infty(\mu)} \leq 1$, the $C^*$-identity yields
\[
\norm{\langle  W^{-1} \rangle^{1/2}_{\mu,Q} \langle |\psi_Q|^2 W \rangle^{1/2}_{\mu, Q} }^2_{\bh} \leq \norm{\langle W \rangle_{\mu, Q}^{1/2} \, \langle W^{-1} \rangle_{\mu, Q}^{1/2}}^2_{\B(\Hil)}.
\]
Consequently, we deduce that
\begin{equation*}\label{Uaverbdd}
\norm{\frac{1}{\mu(Q)}\, \left(\psi_Q \otimes \varphi_Q \right)_\mu }_{L^2_W \to L^2_W} \leq \norm{\langle W \rangle_{\mu, Q}^{1/2} \, \langle W^{-1} \rangle_{\mu, Q}^{1/2} }_{\bh}.
\end{equation*}



\end{proof}
\begin{remark}\thlabel{1Qnorm}
We remark that if $\psi_Q, \varphi_Q$ are both equal to the indicator function $1_Q$, then we actually have the following norm equality
\[
\norm{\frac{1}{\mu(Q)} \left(1_Q \otimes 1_Q \right)}_{L^2_W \to L^2_W }  = \norm{\langle W \rangle_{\mu, Q}^{1/2} \, \langle W^{-1} \rangle_{\mu, Q}^{1/2}}_{\bh}.
\]
Indeed, one can show that the norm-equality is attained using functions of the form $f= 1_Q W^{-1}e$, with $e\in \Hil$.
\end{remark}
\subsection{Sharp estimates for scalar-valued weights}
In this section, we include a couple of auxiliary lemmas about scalar-valued weights, which will be of crucial in the proof of \thmref{SpL2W}. The following result is essentially borrowed from Lemma 4.3 in \cite{NPTV17} and will later allow us to reduce estimates of operator-valued weights to scalar-weights.

\begin{lemma} \thlabel{reductionA2} 
Let $W: \R^d \to \bh$ be an $\A^\S_2(\mu)$-weight and $e\in \Hil$ a non-zero vector. Then $\left( W\, e \, \lvert\, e\right)_{\Hil}$ is a scalar-valued $\A^{\S}_{2}(\mu)$-weight and satisfies 
\begin{equation*} \left[ \left( W  e   \lvert  e \right)_{\Hil} \right]_{\A^\S_2(\mu)} \leq  \left[ W \right]_{\A^\S_2 (\mu)} .
\end{equation*}
\end{lemma}

Reducing inequalities to scalar weights as indicated by \thref{reductionA2}, allows for application of sharp estimates for scalar weights. The following lemma is essentially a quantitative version of the portion preserving property of scalar-valued $A^{\S}_{2}(\mu)$-weights. For the sake of abbreviation, we shall use the following notation $w(E) := \int_E w d\mu$

\begin{lemma}\thlabel{portA2}Let $w$ be a scalar-valued $\A^{\S}_{2}(\mu)$-weight and $0< \delta <1$. Then for every $Q\in \S$ and $S \subset Q$ with $\mu(S) \leq \delta \mu(Q)$, we have that 
\[ 
\int_S w d\mu \leq \left(1- \frac{(1-\delta)^2}{\left[w\right]_{\A^{\S}_2(\mu)}} \right) \int_Q w d\mu
\]
\end{lemma}
\begin{proof}Set $E_S := Q \setminus S$ and notice that $(1-\delta)\mu(Q) \leq \mu(E_S)$. With this at hand, we estimate according to
\[ \int_Q w d\mu \leq \left[ w\right]_{\A^{\S}_2(\mu)}  \frac{\mu(Q)^2}{\int_{E_S} w^{-1} d\mu } \leq  \frac{\left[ w\right]_{\A^{\S}_2(\mu)} }{(1-\delta)^2}\frac{\mu(E_S)^2}{\int_{E_S} w^{-1} d\mu} \leq 
\frac{\left[ w\right]_{\A^{\S}_2(\mu)} }{(1-\delta)^2} \int_{E_S} w d\mu.
\]
The proof readily follows by writing $\int_{E_S}w d\mu= \int_Q w d\mu - \int_S w d\mu$ and rearranging in the previous inequality. 
\end{proof}

In the context of dyadic Muckenhoupt weights, we can actually obtain a sharper version of the portion preserving property, which does not utilize the full strength of the dyadic $A_2$-condition and instead relies on the weaker notion of dyadic $A_{\infty}$-weights. In the context of $\A_\infty(\mu)$, we shall need to assume that $\mu$ is a doubling measure with constant $K_\mu$, previously defined in \eqref{Doublingcond}. We state and prove a tailor-made version of this principle in the following context.

\begin{lemma}\thlabel{sAinf}
Let $w$ be a scalar-valued dyadic $\A_{\infty}(\mu)$-weight and let $0<\delta < 2^{-16K^{12}_\mu [w]_{\A_{\infty}(\mu)}}$. Then for any $Q\in \Dy$ and Borel set $S\subset Q$ with $\mu(S) \leq \delta \mu(Q)$, there exists $0< \eta < \frac{1}{2}$, such that $w(S) \leq \eta w(2Q)$. In fact, we can take $\eta= 2K^2_{\mu} \delta^{\varepsilon/2}$, with $\varepsilon = \frac{1}{6K^{10}_\mu [w]_{\A_\infty(\mu)} } $.
\end{lemma}

\begin{proof} This proof relies on a sharp version of the reverse H\"older inequality (see \cite{hytonen2012sharp}, Theorem 1.1), adapted to our setting. For instance, it asserts that for $0< \varepsilon \leq \frac{1}{ 6K^{10}_\mu [w]_{\A_\infty(\mu)} }$, one has
\begin{equation*} \langle w^{1+\varepsilon} \rangle_{\mu, Q}^{1/(1+\varepsilon)} \leq 2K^2_\mu  \langle w \rangle_{\mu, 2Q}
\end{equation*}
for all $Q\in \Dy$. Now, let $S\subset Q$ with $\mu(S) \leq \delta  \mu(Q)$. By H\"older's inequality and the sharp version the of the reverse H\"older inequality, we get 
\begin{multline*}
\int_S w d\mu \leq  \mu(S)^{\varepsilon/(1+\varepsilon)} \mu(Q)^{1/(1+\varepsilon)}   \langle w^{1+\varepsilon} \rangle_{\mu,Q} ^{1/(1+\varepsilon)} \leq
2 K^2_\mu  \delta^{\varepsilon /(1+\varepsilon)} \int_{2Q} w d\mu \\
\leq 2 K^2_\mu  \delta^{\varepsilon /2} \int_{2Q} w d\mu.
\end{multline*}
It is straightforward to check that $2 K^2_\mu \, \delta^{\varepsilon /2} < \frac{1}{2}$, whenever $0<\delta < 2^{-16K^{12}_\mu [w]_{\A_{\infty}(\mu)}}$.
\end{proof}

%
%

\section{Proof of Main Result} 
\label{MResSec}
%
%

\subsection{The Lower Bound} 
\label{LowA2bdd}

\begin{proof}[Proof of the lower bound of \thmref{SpL2W} ]
Given a sparse collection $\S \subset \Dy$, we let $\sigma: \S \to \mathbb{Z}$ an injective function and consider the standard orthogonal basis on $L^2([0, 2\pi), \C)$ given by the trigonometric system $\mathcal{E}:= \{e^{int} : n\in \mathbb{Z} \}$. With this at hand, we define the sparse Lerner operators
\begin{equation}\label{TrigLerops}
T^{\S, \mathcal{E}}_{\psi, \varphi} f (x) = \sum_{Q\in \S} e^{i\sigma(Q) t} \psi_Q(x) \langle \varphi_Q f \rangle_{\mu , Q}.
\end{equation}
Now suppose there exists a constant $C>0$, possibly depending on $W$, such that
\[
\sup_{\psi, \varphi} \norm{(T^\S_{\psi, \varphi} \otimes \mathds{1}) f}_{L^{2}_{W}}   \leq  C \norm{f}_{L^2_W}
\]
for all $f \in L^2_W$. Since elements in $\mathcal{E}$ are unimodular constants, we also have that all operators $T^{\S, \mathcal{E}}_{\psi, \varphi}$ have operator-norm on $L^2_W$ bounded by $C$. Now using the orthogonality assumption of $\mathcal{E}$, we get that
\[
\sum_{Q\in \S} \norm{ \psi_Q \langle \varphi_Q f \rangle_{\mu,Q} }^2_{L^2_W} = \int_{0}^{2\pi} \norm{ (T^{\S, \mathcal{E}}_{\psi, \varphi}\otimes \mathds{1}) f }^2_{L^2_W} \frac{dt}{2\pi} \leq C^2 \norm{f }^2_{L^2_W}
\]
for all $f\in L^2_W$. In particular, this means that for any $Q \in \S$, we have that
\[
\sup_{\psi, \varphi} \norm{ \psi_Q \langle \varphi_Q f \rangle_{\mu,Q} }_{L^2_W} \leq C \norm{f}_{L^2_W}.
\]
According to \thref{1Qnorm}, we have  
\[
\norm{\langle W \rangle^{1/2}_{\mu, Q} \langle W^{-1} \rangle^{1/2}_{\mu, Q} }_{\bh} = \sup_{\norm{f}_{L^2_W}= 1} \norm{ 1_Q \langle 1_Q f \rangle_{\mu, Q} }_{L^2_W} \leq C.
\]
Taking supremum over $Q\in \S$ and infimum over all constant $C>0$, we finally conclude that 
\[
[W]^{1/2}_{\A^\S_2(\mu)} \leq \sup_{\psi, \varphi} \norm{(T^{\S}_{\psi, \varphi} \otimes \mathds{1} ) }_{L^2_W \to L^2_W}.
\]
\end{proof}

We remark that if a single sparse Lerner operator is bounded on $L^2_W$, then have only managed to obtain the lower bound of its operator norm in terms of $[W]^{1/4}_{\A^\S_2(\mu)}$, which is slightly weaker. We are not certain whether this can be sharpened.

\subsection{The Upper Bound}

In order to establish the upper bound of \thmref{SpL2W}, we need to verify that the sequence of linear operators $\left\{T_{n}\right\}_{n\in \mathbb{Z}}$ in \eqref{aoterm} satisfy the prerequisites of the \thref{gencotlar}. For the sake of abbreviation, we shall simply denote by $T_{n}$ the unique canonical extensions to $L^2_W$.

\begin{proof}[Proof of the upper bound of \thmref{SpL2W}]
Fix an arbitrary sequence of bounded complex-valued functions $\{\psi_Q \}, \{\varphi_Q\}$, where each of the $\psi_Q, \varphi_Q$ are supported on $Q$ and satisfy $\norm{\psi_Q}_{L^\infty(\mu)}, \norm{\varphi_Q}_{L^\infty(\mu)} \leq 1$. Recall that the decomposition in subsection \ref{PAOsec} allows us to express an arbitrary sparse Lerner operator as 
\[
T^\S_{\psi, \varphi} = \sum_{n \in \mathbb{Z}} T_n
\]
where 
\[
T_n := \sum_{Q \in \mathcal{J}^n} \frac{1}{\mu(Q)} \left( \psi_Q \otimes \varphi_Q \right)_\mu 
\]
We shall now verify that the hypothesis of \thref{gencotlar} are satisfied. Note that since $\mathcal{J}^{n}$ is a disjoint collection, for all $n\in \mathbb{Z}$, we can apply \thref{diag}, which gives

\begin{multline}\label{Diagest}
\norm{T_{n}f}^{2}_{L^{2}_{W}} = \sum_{Q \in \mathcal{J}^{n}} \norm{\psi_{Q} \langle \varphi_Q f \rangle_{\mu, Q}}^{2}_{L^{2}_{W}} \leq \left[ W \right]_{\A^\S_2(\mu)} \sum_{Q\in \mathcal{J}^n}  \norm{1_{Q}  f}_{L^2_W}^2 \\ \leq  \left[ W \right]_{\A^\S_2(\mu)}  \norm{f}^2_{L^2_W}.
\end{multline} 
This yields the following uniform bound for all the diagonal terms 
\begin{equation} \norm{T^{*}_{n} \, T_{n}}_{L^{2}_{W} \rightarrow L^{2}_{W}} =  \norm{T_{n} \, T^{*}_{n}}_{L^{2}_{W}\rightarrow L^{2}_{W}} =  \norm{T_{n}}^{2}_{L^{2}_{W} \rightarrow L^{2}_{W}} \leq \, \left[W\right]_{\A^\S_2(\mu)} .
\end{equation}
It now remains to establish bounds for the non-diagonal terms. To this end, we may without loss of generality assume that $n>m$, since by the $C^{*}$-identity, we can write

\begin{equation*} \norm{T^{*}_{n} \, T_{m}}_{L^{2}_{W}\rightarrow L^{2}_{W}} \, = \, \norm{T^{*}_{m} \, T_{n}}_{L^{2}_{W}\rightarrow L^{2}_{W}}.
\end{equation*}
Now observe that the boundedness of $T_{n}: L^{2}_{W}\rightarrow L^{2}_{W}$ is equivalent to the boundedness of the following composition of operators on the non-weighted $\Hil$-valued $L^{2}$-space (think $L^2_W$ with $W$ being the identity operator on $\Hil$), namely  

\begin{equation*}L_n := M_{W^{1/2}}\, T_{n} \, M_{W^{-1/2}} : L^2 \to L^2
\end{equation*} 
where $M_{W^{\pm 1/2}}$ denotes the usual multiplication operator with $W^{\pm 1/2}$. This reduction is obvious, as illustrated by the following diagram 
%
%
\[ \begin{tikzcd}[column sep=1in,row sep=0.4in]
L^{2} \arrow{r}{ M_{W^{1/2}}\, T_{n} \, M_{W^{-1/2}} } \arrow[swap]{d}{M_{W^{-1/2}}} & L^{2} \\
L^{2}_{W}\arrow{r}{T_{n}} & L^{2}_{W} \arrow[swap]{u}{M_{W^{1/2}}} 
\end{tikzcd}
\]
Now we can easily compute the adjoint of the operator $L_n$ on $L^{2}$, which is given by
\[
L^*_n = M_{W^{-1/2}} T^*_n M_{W^{1/2}} = M_{W^{-1/2}} \sum_{Q \in \mathcal{J}^n} \A^{*}_{\mu, Q} M_{W^{1/2}}.
\]
where the adjoint of the averaging operator on $L^2$ takes the form 
\[
\A^{*}_{\mu, Q}(f)(x) = \frac{1}{\mu(Q)} \conj{\varphi_Q(x)}\int_{Q}\conj{\psi_Q(y)}f(y) d\mu(y).
\]
By the $C^*$-identity, we actually seek a bound for the operator norm of 
\begin{equation}\label{eqnorm1}
\norm{T_{n}^{*} T_{m}}^{2}_{L^2_W \to L^2_W} = \norm{L^*_n L_m }^2_{L^2 \to L^2}.
\end{equation}
Using the fact that $\mathcal{J}^n$ is a disjoint collection, we can explicitly compute the kernel expression of the positive linear operator $L_n L^*_n= M_{W^{1/2}}T_n M_{W^{-1}} T^*_n M_{W^{1/2}}$, given by 
\[
L_n L^*_n = \sum_{Q \in \mathcal{J}^n} \left( M_{W^{1/2}} \A_{\mu, Q} M_{W^{-1/2}} \right)(M_{W^{1/2}} \A_{\mu, Q} M_{W^{-1/2}})^*.
\]
As a consequence, using the $C^*$-identity, we can write
\[
\norm{L^*_n L_m f}^2_{L^2 \to L^2} = \left( L_n L^*_n L_m f \lvert L_m f \right)_{L^2} = \sum_{Q \in \mathcal{J}^n} \norm{(M_{W^{1/2}} \A_{ \mu, Q} M_{W^{-1/2}})L_m f}^2_{L^2}.
\]
According to \thref{diag}, the family of averaging operators  $M_{W^{1/2}} \A_{\mu, Q} M_{W^{-1/2}}$ are localized at $Q$ and map $L^2 \to L^2$ with operator norm uniformly bounded by $[W]^{1/2}_{\A^{\S}_2(\mu)}$. Using this we obtain 
\begin{equation}\label{L_nest}
    \norm{L^*_n L_m f}^2_{L^2} \leq \left[W \right]_{\A^{\S}_2(\mu)} \sum_{Q \in \mathcal{J}^n} \norm{1_Q L_m(f)}^2_{L^2}.
\end{equation}
It remains to estimate the sum on the right hand side of \eqref{L_nest}. To this end, recall that $n>m$ and $\mathcal{J}^m$ disjoint, hence changing the order of summation, it is straightforward to check that we can express 
\begin{multline} \label{Mixest}
    \sum_{Q \in \mathcal{J}^n} \norm{1_Q L_m(f)}^2_{L^2} = \sum_{R\in \mathcal{J}^m} \sum_{Q \in \mathcal{J}^{(n-m)}(R)} \norm{1_Q \left( M_{W^{1/2}} \A_{\mu, Q} M_{W^{-1/2}} \right)f}^2_{L^2} \\
    = \sum_{R\in \mathcal{J}^m} \sum_{Q \in \mathcal{J}^{(n-m)}(R)} \int_Q \abs{\psi_R(x)}^2 \left( W(x) \langle \varphi_R W^{-1/2}f \rangle_{\mu, R} \lvert \langle \varphi_R W^{-1/2}f \rangle_{\mu, R} \right)_{\Hil} d\mu(x) \\
    \leq \sum_{R\in \mathcal{J}^m} \sum_{Q \in \mathcal{J}^{(n-m)}(R)} \int_Q \left( W(x) \langle \varphi_R W^{-1/2}f \rangle_{\mu, R} \lvert \langle \varphi_R W^{-1/2}f \rangle_{\mu, R} \right)_{\Hil} d\mu(x).
\end{multline}
Note that in the last step, we used the simple estimate $\norm{\psi_{R}}_{L^\infty(\mu)} \leq 1$. Let $e=\langle \varphi_R W^{-1/2} f \rangle_{\mu, R}$ and consider the scalar-weight $w_e(x) = (W(x) e \lvert e)_\Hil$. Combining the previous estimates in \eqref{L_nest} and in \eqref{Mixest}, we can write
\begin{equation} \label{est2}
\norm{L^*_n L_m f}^2_{L^2}\leq 
[W]_{\A^\S_2(\mu)} \sum_{R\in \mathcal{J}^m } \sum_{Q \in \mathcal{J}^{(n-m)}(R) } \int_Q w_e(x) d\mu(x).
\end{equation}
Recall that by \thref{reductionA2}, $w_e$ is a scalar $A^{\S}_{2}$-weight with $\left[w_e\right]_{A^\S_2(\mu)} \leq \left[W\right]_{\A^\S_2(\mu)}$ (independent $e\in \Hil$). According to the \textit{decaying stopping time} property in \eqref{DSTn} we have 

\begin{equation*} \sum_{Q \in \mathcal{J}^{(n-m)}(R)} \mu(Q)  \leq \left(\frac{1}{2}\right)^{(n-m)}  \mu(R) .
\end{equation*}
Regrouping all cubes in $\mathcal{J}^{(n-m)}(R)$ with common predecessors in $\mathcal{J}^{(n-m-1)}(R)$ and successively applying \thref{portA2} in each step as we move towards the top cube $R$, we obtain

\begin{equation*}\label{A2Sest} 
\sum_{Q \in \mathcal{J}^{(n-m)}(R)} \int_Q w_e d\mu  \leq \left(1-\frac{1}{4\left[W\right]_{\A^\S _2(\mu)} } \right)^{(n-m)} \int_R w_e d\mu.
\end{equation*}
Going back to the estimate in \eqref{est2} and using these observations with $e=\langle \varphi_R W^{-1/2} f \rangle_{\mu, R}$ and applying \thref{diag}, we obtain

\begin{multline*}
\norm{L^*_n L_m f}^2_{L^2}  \leq [W]_{\A^\S_2(\mu)} \left(1-\frac{1}{4\left[W\right]_{\A^{\S}_2(\mu) }} \right)^{(n-m)} \sum_{R\in \mathcal{J}^{m}} \int_R w_e d\mu \\ 
=  [W]_{\A^\S_2(\mu)} \left(1-\frac{1}{4\left[W\right]_{\A^\S _2(\mu)} } \right)^{(n-m)}  \sum_{R \in \mathcal{J}^m} \norm{ \frac{1}{\mu(R)} \left( 1_R \otimes \varphi_R \right)_\mu W^{-1/2} f }^2_{L^2_W}  \\
\leq [W]^2_{\A^{\S}_2(\mu)} \left(1-\frac{1}{4\left[W\right]_{\A^\S _2(\mu)} } \right)^{(n-m)} \sum_{R \in \mathcal{J}^m } \norm{1_R W^{-1/2}f }^2_{L^2_W} \\ 
\leq [W]^2_{\A^{\S}_2(\mu)} \left(1-\frac{1}{4\left[W\right]_{\A^\S _2(\mu)} } \right)^{(n-m)} \norm{f}^2_{L^2}.
\end{multline*}

Now recalling the identity in \eqref{eqnorm1}, we arrive at 

\begin{equation*} \norm{T_{n}^{*} \, T_{m} }_{L_{W}^{2}\rightarrow L^{2}_{W}} \leq  \left(1-\frac{1}{4\left[W\right]_{\A^\S _2} } \right)^{(n-m)/2} \, \left[ W\right]_{\A^\S_2 (\mu)}.
\end{equation*}
In an identical manner, we can estimate the operator norms 

\[ \norm{T_{n} T_{m}^{*}}_{L^2_W \to L^2_W}  =  \norm{L_n L^*_m}_{L^2 \to L^2}.
\] 
Indeed, this is done by running through the same argument as before, with the exception of the dual weight $W^{-1}$ playing the previous role of $W$. Since the $\A^\S_2(\mu)$-condition is symmetric, that is $\left[W^{-1}\right]_{\A^{\S}_2(\mu)} = \left[W\right]_{\A^{\S}_2(\mu)}$, the proof principally remains unchanged. We then analogously get 

\[ \norm{T_{n}\, T_{m}^{*}}_{L^2_W \to L^2_W} \leq \left(1-\frac{1}{4\left[W\right]_{\A^\S _2(\mu)} } \right)^{(n-m)/2} \, \left[ W\right]_{\A^\S_2 (\mu)}.
\] 
The hypothesis of \thref{gencotlar} are thus satisfied, and  an application gives

\begin{multline*} 
\norm{\sum_{n \in \mathbb{Z} }  T_n }_{L^2_W \to L^2_W} \leq 2 \sum_{n\in \mathbb{Z}} [W]^{1/2}_{\A^\S_2(\mu)} \left(1-\frac{1}{4[W]_{\A^\S _2(\mu)} } \right)^{\abs{n}/4} 
\leq \frac{4[W]^{1/2}_{\A^{\S}_2(\mu)}}{1-\left(1-\frac{1}{4[W]_{\A^\S_2(\mu)}} \right)^{1/4}}. 
\end{multline*}
According to \thref{gencotlar} again, we also have that $\sum_n T_n$ converges unconditionally to $T^{\S}_{\psi, \varphi}$. Using the simple inequality $\frac{1}{1-t^{1/4}} = \frac{(1+t^{1/4})(1+t^{1/2})}{1-t} \leq \frac{4}{1-t}$, for $0<t <1$, we finally conclude that 

\begin{equation}\label{FineqC1} \sup_{\psi, \varphi} \norm{ T^\S_{\psi, \varphi} \otimes \mathds{1} }_{L^2_W \to L^2_W} \leq 64 \left[W \right]^{3/2}_{\A^\S_2(\mu)}.
\end{equation}
Together with the lower bound in subsection \ref{LowA2bdd}, the proof of \thref{SpL2W} is complete. 

\end{proof}
We now turn to the proof of \thref{MixbddCor}.
%
%
\begin{proof}[Proof of \thref{MixbddCor}]
This time, we assume that $\mu$ is a doubling measure with constant $K_\mu$ and that $W$ is a dyadic Muckenhoupt $\A_2(\mu)$-weight. The proof follows that of \thref{SpL2W} verbatim, up until the step in \eqref{est2}, thus it suffices to continue from there. Again, set $w_e = (We \lvert e)_\Hil$ and for each $R\in \mathcal{J}^m$, let $\widehat{R}$ denote the unique child of $R$ in $\S$, which contains $\bigcup_{Q \in \mathcal{J}^{(n-m)}(R) }  Q $. Now recall the decaying stopping time property in \eqref{DSTn} saying that
\[
\mu\left( \bigcup_{Q \in \mathcal{J}^{(n-m)}(R) }  Q \right) = \sum_{Q \in \mathcal{J}^{(n-m)}(R)} \mu(Q) \leq 2^{-(n-m)}\mu(R).
\]
If $(n-m) > 16 K^{12}_\mu [W]_{\A_\infty (\mu)} $, then we may apply \thref{sAinf} to the sets $\bigcup_{Q \in \mathcal{J}^{(n-m)}(R) }  Q$ and $\widehat{R}$, giving 

\begin{equation*}\sum_{Q\in \mathcal{J}^{(n-m)}(R)}  \int_Q w_e d\mu \leq  \eta_{+}(n-m) \int_{2\widehat{R}} w_e d\mu \leq  \eta_{+}(n-m) \int_R w_e d\mu
\end{equation*}
where $\eta_{+}(n-m) = 4K^2_\mu 2^{-(n-m) \varepsilon_+ }$ with $\varepsilon_+ = 1/(12K^2_\mu [W]_{\A_\infty(\mu)})$. Now going back to \eqref{est2} with $e = \langle \varphi_R W^{-1/2} f \rangle_{\mu, R}$, we have for $(n-m) > 16 K^{12}_\mu [W]_{\A_\infty (\mu)}$:

\begin{multline*}
\norm{L^*_m L_n f}^2_{L^2}  \leq [W]_{\A_2(\mu)} \eta_+(n-m) \sum_{R\in\mathcal{J}^{m}} \int_R w_e d\mu \\
\leq  [W]_{\A_2(\mu)} \eta_+(n-m) \sum_{R \in \mathcal{J}^m} \norm{ \frac{1}{\mu(R)} \left( 1_R \otimes \varphi_R \right)_\mu W^{-1/2} f }^2_{L^2_W}  \\
\leq [W]^2_{\A_2(\mu)} \eta_+(n-m) \sum_{R \in \mathcal{J}^m } \norm{1_R W^{-1/2}f }^2_{L^2_W} 
\leq [W]^2_{\A_2(\mu)} \eta_+(n-m) \norm{f}^2_{L^2}.
\end{multline*}
Note that in the intermediate step we also used \thref{diag}. In light of \eqref{eqnorm1}, we thus arrive at

\begin{equation}\label{opnormest1} 
\norm{T_{n}^{*}  T_{m} }_{L_{W}^{2}\rightarrow L^{2}_{W}} \leq  \eta_{+}(n-m)^{1/2}  \left[ W\right]_{\A_2(\mu)} \, \, , \text{if} \, \, (n-m) > 16 K^{12}_\mu [W]_{\A_\infty (\mu)}.
\end{equation}

Now if $(n-m) \leq 16K^{12}_\mu [W]_{\A_\infty (\mu)}$, we simply use the trivial estimate

\begin{equation*}\sum_{Q\in \mathcal{J}^{(n-m)}(R)}  \int_Q w_e d\mu   \leq  \int_R w_e d\mu. 
\end{equation*}
Similar but simpler calculations then gives
\begin{equation}\label{opnormest2}\norm{T_{n}^{*} T_{m} }_{L_{W}^{2}\rightarrow L^{2}_{W}} \leq  \left[ W\right]_{\A_2(\mu)} \qquad , \text{if} \qquad (n-m) \leq 16K^{12}_\mu [W]_{\A_\infty (\mu)}.
\end{equation}
Combining \eqref{opnormest1}, \eqref{opnormest2} and recalling that $n>m$, we arrive at

\[ \norm{T_{n}^{*} \, T_{m} }_{L_{W}^{2}\rightarrow L^{2}_{W}}  \leq \alpha(n-m ) 
\]
where 

\[ \alpha(n-m) :=
\begin{cases} 
      \left[W\right]_{\A_2(\mu)} \,  &  \, \abs{n-m} \leq 16K^{12}_\mu [W]_{\A_\infty (\mu)} \\
       \left[ W\right]_{\A_2(\mu)} \eta_{+}(\, \abs{n-m}\, )^{1/2} &  \, \abs{n-m} > 16K^{12}_\mu [W]_{\A_\infty (\mu)}.
   \end{cases}
\]
In an identical manner, we can estimate the norm $ \norm{T_{n}\, T_{m}^{*}}_{L^{2}_{W}\rightarrow L^{2}_{W}} $. Again, we repeat the same procedure as before, with the exception of substituting $W$ with its dual weight $W^{-1}$ and using the symmetry that $W\in \A_2(\mu)$ iff $W^{-1}\in \A_2(\mu)$. In this case, we apply \thref{sAinf} to $\sigma_e = (W^{-1}e \lvert e)_\Hil$ which gives rise to the parameters $\varepsilon_{-} = 1/(12K^2_\mu [W^{-1}]_{\A_\infty(\mu)})$ and $\eta_{-}(n-m) = 4K^2_\mu 2^{-(n-m) \varepsilon_- }$. At the end, we analogously arrive at
\[
 \norm{T_{n} \, T_{m}^{*}}_{L^2_W \to L^2_W}  \leq \beta(n-m ) 
\]
where 
\[ \beta(n-m) :=
  \begin{cases} 
      \left[W\right]_{\A_2(\mu)} \,  &  \, \abs{n-m} \leq 16K^{12}_\mu [W^{-1}]_{\A_\infty (\mu)} \\
      \left[ W\right]_{\A_2(\mu)} \eta_{-}(\, \abs{n-m})^{1/2}  & \, \abs{n-m} > 16K^{12}_\mu [W^{-1}]_{\A_\infty (\mu)}.
   \end{cases}
\]
\thref{gencotlar} on the principle of almost orthogonality now applies and we deduce 

\begin{equation}\label{AppCot} 
\norm{T^\S_{\psi, \varphi} \otimes \mathds{1}}_{L^2_W \to L^2_W} = \norm{ \sum_{n\in \mathbb{Z}}  T_{n}  }_{L^2_W \to L^2_W}  \leq  2 \left(\sum_{n\in \mathbb{Z}}  \sqrt{\alpha(n)}  \cdot  \sum_{n\in \mathbb{Z}} \sqrt{\beta(n)} \right)^{1/2}.
\end{equation}
It now remains to estimate the sums on the right hand side of \eqref{AppCot}. We write

\begin{equation*}
\sum_{n=1}^{\infty} \sqrt{\alpha(n)} = [W]^{1/2}_{\A_2(\mu)} \left(
\sum_{\abs{n} \leq 16K^{12}_\mu [W]_{\A_\infty (\mu)}}   1 + \sum_{\abs{n}> 16K^{12}_\mu [W]_{\A_\infty (\mu)}}  \eta_{+}(|n|)^{1/4} \right) =: S_{1} + S_{2}.
\end{equation*}
The first sum is trivially bounded by

\begin{equation*}S_{1}  \leq  32K^{12}_\mu [W]^{1/2}_{\A_2(\mu)} [W]_{\A_\infty (\mu)}
\end{equation*}
Recalling that $\eta_+(|n|) = 4K^2_\mu 2^{-|n| \varepsilon_+}$ with $\varepsilon_+ = 1/(12K^{12}_\mu [W]_{\A_\infty(\mu)})$, the second sum is a geometric series and can be estimated according to

\begin{equation*} S_{2}  \leq  2 [W]^{1/2}_{\A_2(\mu)} \sum_{n=0}^{\infty} \eta_+ (n)^{1/4} = \frac{8K^2_\mu [W]^{1/2}_{\A_2(\mu)}}{1-2^{-\varepsilon_+/4}}.
\end{equation*}
Note that the function $t\mapsto \frac{t}{1-2^{-t/4}}$ is bounded on $[0,1]$, hence we can find a numerical constant $c>0$ ($c=7$ will do), such that 
\[
S_2 \leq  8 K^2_\mu [W]^{1/2}_{\A_2(\mu)} c/\varepsilon_+ = c_{\mu} [W]^{1/2}_{\A_2(\mu)} [W]_{\A_\infty(\mu)}.
\]
Here $c_\mu >0$ is a constant only depending $\mu$, which we shall allow to change from line to line. Adding up the estimates of $S_{1}$ and $S_{2}$ yields

\begin{equation*} \sum_{n\in \mathbb{Z}} \sqrt{\alpha(n)} \leq c_\mu \left[W\right]^{1/2} _{\A_2(\mu)} \, \left[W\right]_{\A_{\infty}(\mu)} .
\end{equation*} 
An identical argument also shows that 

\begin{equation*}  \sum_{n\in \mathbb{Z}} \sqrt{\beta(n)} \leq c_\mu \left[W\right]^{1/2} _{\A_2(\mu)} \left[W^{-1}\right]_{\A_{\infty}(\mu)}.
\end{equation*}
Going back to \eqref{AppCot} with these estimates at hand, we ultimately arrive at

\begin{equation*} \sup_{\psi, \varphi, \S} \norm{T^{\S}_{\psi,\varphi} \otimes \mathds{1}}_{L^2_W \to L^2_W}  \leq  c_\mu \left[W\right]^{1/2}_{\A_2(\mu)}  \left[W\right]^{1/2}_{\A_{\infty}(\mu)}  \left[W^{-1} \right]^{1/2}_{\A_{\infty}(\mu)}.
\end{equation*}
This completes the proof of \thref{MixbddCor}.
\end{proof}
%
%
%
%
\section{weighted bounds for the Bergman projection}
\label{Bergsec}
In this section, we shall prove \thref{BergL2W}, based on a convex body domination by sparse operators. To this end, we shall need to introduce the following dyadic grids on $\R$
\[
\Dy^\omega(\R) := \{ 2^j ( [0,1) + m + (-1)^j \omega ): m \in \mathbb{Z}, \, j \in \mathbb{Z} \}
\]
for $\omega \in \{0,1/3 \}$. These systems have previously appeared in many different works on sparse domination, see for instance \cite{Ler13}, \cite{HytPer13} and references therein. Note that $\Dy^0(\R)$ is just the standard grid on $\R$, while $\Dy^{1/3}(\R)$ is a shifted alternating grid, but when combined, they have the following useful property.

\begin{lemma}[Lemma 3.1, \cite{PePo13}] \thlabel{Dygridlem} 
For any interval $I \subset \R$, there exists a dyadic interval $J \in \Dy^\omega(\R)$ for some $\omega \in \{0,1/3\}$, such that $I \subset J$ and $|J| \leq 8 |I|$.
\end{lemma}
Using these dyadic grids, we shall consider the corresponding collections of Carleson squares
\[
\mathcal{Q}^\omega := \left\{ Q_J := J \times (0 , |J| ] : J \in \Dy^\omega(\R) \right\}
\]
with $\omega \in \{0,1/3\}$, which are easily seen to be sparse collection of dyadic cubes on $\C_+$. For a fixed $f \in L^\infty_0 (dA_\gamma) \otimes \Hil$, we introduce the convex body averages
\[
\llangle f \rrangle_{\gamma, Q_J} := \left\{ \frac{1}{A_\gamma(Q_J)} \int_{Q_J} \varphi f dA_\gamma : \, \varphi: Q_J \to \C, \, \norm{\phi}_{\infty} \leq 1 \right\}.
\]
The convex body averages are symmetric, convex and compacts subsets of $\Hil$, where the compactness follows from weak-compactness of the unit-ball of $\Hil$ together with the weak-star compactness of the unit-ball of $L^\infty(dA_\gamma)$. With this at hand, we can define the set-valued sparse operators 
\begin{equation}\label{Minkowsum}
L_\gamma f(z) = \sum_{\omega \in \{0, 1/3\}} \sum_{J \in \Dy^\omega} 1_{Q_J}(z) \llangle f \rrangle_{Q_J, \gamma}
\end{equation}
regarded as Minkowski sums of convex body averages. It follows from arguments identical to Lemma 2.5 in \cite{NPTV17}, that the corresponding Minkowski sum in \eqref{Minkowsum} is for each $f\in L^\infty_0 (dA_\gamma) \otimes \Hil$ and a.e $z\in \C_+$, is a bounded symmetric convex set of $\Hil$. The main reason for introducing these operators, is show that for each fixed $z\in \C_+$ and $f \in L^\infty_0(dA_\gamma)\otimes \Hil$, the family of maximal Bergman projections $P^+_\gamma$ with $\gamma>-1$, belongs to a fixed dilation of the set \eqref{Minkowsum}. The content of our next result is the following convex body domination, using techniques inspired from \cite{PePo13}. 

\begin{prop} \thlabel{CBDprop}
There exists a constant $C_\gamma>0$, only depending on $\gamma > -1$ such that for any $f\in L^\infty_0(dA_\gamma)\otimes \Hil$ and $z \in \C_+$, we have 
\[
P^{(+)}_{\gamma}(f)(z) \in C_\gamma L_\gamma (f)(z).
\]
\end{prop}
\begin{proof}
It suffices to prove the claim for the maximal Bergman projections $P^+_\gamma$. To this end, fix an arbitrary $z \in \C_+$ and $f\in L^\infty_0(dA_\gamma)\otimes \Hil$. Note that we can write
\[
P^+_\gamma f(z) = \sum_{k= -\infty}^{\infty} \int_{2^k \leq |z-\conj{\xi}| <2^{k+1} } \frac{f(\xi)}{|z-\conj{\xi}|^{2+\gamma}} dA_\gamma(\xi).
\]
It thus suffices to find a constant $C_\gamma>0$, such that each term satisfies 
\begin{equation}\label{Cbterm}
\int_{2^k \leq |z-\conj{\xi}| <2^{k+1} } \frac{f(\xi)}{|z-\conj{\xi}|^{2+\gamma}} dA_\gamma(\xi) \in C_\gamma \llangle f \rrangle_{Q_{J_k}, \gamma}
\end{equation}
for some $J_k \in \Dy^\omega(\R)$ with $\omega \in \{0,1/3\}$ and the collection $\{J_k\}_{k\in \mathbb{Z}}$ has finite multiplicity. To this end, fix an integer $k$ and pick an arbitrary $\xi \in \C_+$ satisfying $2^k \leq |z-\conj{\xi}| <2^{k+1}$. If $\Re(z) \leq \Re(\xi)$, then $\xi \in Q_{I(z)}$ where $I(z) := [ \Re(z), \Re(z) + 2^{k+1})$, and if $\Re(z) > \Re(\xi)$ then we instead pick $I(z):= [\Re(z)-2^{k+1}, \Re(z) )$ for which $\xi \in Q_{I(z)}$. According to \thref{Dygridlem}, there exists an interval $J_k \in \Dy^{\omega}$, for some $\omega \in \{0,1/3\}$, such that $I(z) \subset J_k$ and $|J_k| \leq 8 |I(z)|$. From this it follows that 
\[
1_{\{\xi: 2^k \leq |z - \conj{\xi} | < 2^{k+1} \} }(z) \frac{1}{|z- \conj{\xi}|^{2+\gamma}} \leq 1_{Q_{J_k}}(z) 2^{-k(2+\gamma)} \leq 2^{2+\gamma} 8^{2+\gamma} 1_{Q_{J_k}}(z) \frac{1}{|J_k|^{2+\gamma}}.
\]
This establishes \eqref{Cbterm} with $C_\gamma = 2^{2+\gamma}8^{2+ \gamma}= 16^{2+\gamma}$, thus it suffices to prove that each interval $J_k$ in \eqref{Cbterm} appears at most finitely many times. However, note that by construction each interval $J_{k}$ contains an interval of length $2^{k+1}$ and is itself of length no more than $2^{k+4}$, thus for any pair of integers $k,m \in \mathbb{Z}$ with $\abs{m-k} > 4$, the intervals $J_k, J_m$ must necessarily be distinct. Consequently the collection of intervals $\{J_k\}_{k \in \mathbb{Z}}$ have at most multiplicity $4$, which completes the proof of this proposition.
\end{proof}

\thref{CBDprop} tells us that in order to find weighted bounds for the (maximal) Bergman projections $P^{(+)}_\gamma$, it suffices to understand the set-valued sparse operators in \eqref{Minkowsum}. Although, these set-valued operators are complicated objects, Lemma 2.7 in \cite{NPTV17} tells us that it suffices to find a uniform bound for the following family of sparse operators 
\[
\widetilde{T}_{\S}(f)(x) = \sum_{Q \in \S}\underbrace{\frac{1}{\mu(Q)} \int_Q \kappa_Q(x,y) f(y) d\mu(y)}_{:= \A_{\mu,\kappa_Q}(f)(x)}
\]
where the $\kappa_Q(x,y)$ are allowed to be any complex-valued kernels supported on $Q\times Q$ and belonging to the unit-ball of $L_0^\infty(\mu \otimes \mu)$. Unfortunately, we are not certain that the operator norm of $\A_{\mu,\kappa_Q}: L^2_W(\Hil, \mu) \to L^2_W(\Hil, \mu)$ is uniformly bounded by $[W]^{1/2}_{\A^{\S}_2(\mu)}$, unless $\Hil \cong \C^N$, in which the equivalence of norms on $\Hil$ plays a crucial role in the proof. In fact, if $\Hil \cong \C^N$, then we shall see that it actually suffices to find a uniform bound for the family of sparse Lerner operators in \eqref{Lernerops}. This observations is a consequence a general lemma, which initially appeared in an earlier pre-print version of \cite{NPTV17}. For convenience, we shall phrase it in our context and provide a short sketch of proof.

\begin{lemma}[\cite{NPTV17}] \thlabel{LemLerops}
Let $f \in L^\infty_{0}(dA_\gamma, \C^N)$. Then there exists complex-valued measurable functions $\{\varphi_j \}_{j=1}^{N}$ supported on $Q_J$ with $\norm{\varphi_j}_{\infty} \leq 1$ such that for any $g(\xi) \in \llangle f \rrangle_{Q_J, \gamma}$ for a.e $\xi \in Q_J$, there exists complex-valued measurable functions $\{\psi_j \}_{j=1}^N$, with $\norm{\psi_j}_\infty \leq C_N$, such that 
\[
g(\xi) = \sum_{j=1}^N \psi_j(\xi) \langle \varphi_j f \rangle_{Q_J,\gamma}.
\]
\end{lemma} 
\begin{proof}
Since $K=\llangle f \rrangle_{Q_J, \gamma}$ is convex, there exists a unique ellipsoid $\E_K$ of maximal volume contained in $K$, called the John-ellipsoid of $K$, at it satisfies the property $\E_K \subseteq K \subseteq \sqrt{N} \E_K$. Let $1\leq M \leq N$ denote the dimension of the principal axis of $\E_K$ and $\{e_j\}_{j=1}^M$ denote the vectors corresponding to its principal axis. Since with $e_j \in \E_K \subseteq K$, there exists complex-valued $\varphi_j$ supported on $Q_J$ with $ \norm{\varphi_j}_{L^\infty(\mu)} \leq 1$, such that $e_j = \langle \varphi_j f \rangle_{\mu, Q_J}$, for $j=1, \dots, M$. Now since $K \subseteq \sqrt{N} \E_K$, every measurable vector function $g$ on $Q$ with values in $K$ has the form 
\[
g(x) = \sum_{j=1}^M \psi^g_j(x)e_j \qquad \mu- \text{a.e} \,\, \, x \in Q
\]
where $\{ \psi^g_j(x)\}_{j=1}^M$ are measurable functions, and $\sum_{j=1}^M \abs{\psi^g_j(x)}^2 \leq C(N)$, for some dimensional dependent constant $C(N)>0$. This completes the proof.

\end{proof}

Now \thref{BergL2W} is readily follows from this lemma, in conjunction with \thref{CBDprop} and \thref{SpL2W}. A sparse domination bound for the (maximal) Bergman projections in infinite dimensions seems to require a more precise dyadic model than that of \thref{CBDprop}, which we unfortunately have not been able to find. 

%
%

\section{commutators of sparse operators and applications to the multi-parameter setting} 
\label{AppSec}

In this section, we shall use our main results to show new boundedness results on commutators of sparse operators with operator-valued functions. This in turn can be applied to prove boundedness results for iterated commutators in the bi-parameter setting.

For a locally weakly integrable function $B: \R^d \to \bh$ and a collection of dyadic cubes $\S$, we define the strong operator BMO-norm relative 
to $\S$ by

\begin{multline*}
   \| B \|^2_{SBMO^\S} := \sup_{Q \in \S}  \sup_{ \| e\|_\Hil =1} \frac{1}{|Q|} \int_Q \| B(x)e - \langle B e \rangle_Q \|^2 dx +\\
       \sup_{Q \in \S}  \sup_{\| e\|_\Hil =1} \frac{1}{|Q|} \int_Q \| B^*(x)e - \langle B^* e\rangle_Q  \|^2.
\end{multline*}

We shall denote the space of all such functions with finite norm by $SBMO^\S$. In case that $\S$ is the collection of all dyadic cubes in $\R^d$, we simply write $SBMO^\Dy$. For any Banach space $X$, we write $BMO(\R^d,X)$ for the so-called norm-BMO space consisting of all locally Bochner integrable functions

\[
  f: \R^d \rightarrow X, \quad \sup_{Q \subset \R^d ,\text{ cube}} \frac{1}{|Q|}\int_Q \|f(x) - \langle f \rangle_Q \|^2_X dx < \infty,
\]

and we write $BMO^\Dy(\R^d,X)$ in case the supremum is only taken over dyadic cubes $\Dy$.
It is well-known that the John-Nirenberg Theorem holds in this context, so the $L^2(X)$ norm can be replaced by any $L^p(X)$ norm for $1 < p < \infty$. It is also well-known that $BMO^\Dy(\B(\Hil))$ is strictly contained in $SBMO^\Dy$.Here is the main result of this section.

\begin{thm}   \thlabel{thm:com}
 Let $B: \R^d \rightarrow \B(\Hil)$ be a locally weakly integrable function, let $\S$ be a sparse collection of dyadic cubes in $\R^d$ and $T^\S_{\psi, \varphi}$ be the corresponding
 sparse operator as in \eqref{Lernerops}. Then the family of commutators $[T^\S_{\psi, \varphi} , B]$ given by
 \begin{equation}    \label{def:comm}
      [T^\S_{\psi, \varphi}, B] f  = T^\S_{\psi, \varphi} B f - B T^\S_{\psi, \varphi} f
 \end{equation}
 for $\Hil$-valued functions $f$ with finite Haar expansion,
 extends a bounded linear operator on $L^2(\R^d, \Hil)$, if and only if $B \in SBMO^\S$. In this case, 
 
 \[
  \| [T_\S, B] \|_{L^2(\R^d,\Hil) \to L^2(\R^d,\Hil)} \approx \|B \|_{SBMO_\S}.
 \]
In particular, if $B \in SBMO^\Dy$, then any commutator with a sparse operator as in \eqref{def:comm} is bounded on  $L^2(\R^d, \Hil)$.
\end{thm} 

Before proving this theorem, we shall first establish a corollary about iterated commutators in the biparameter setting, which requires one more definition.
Let $b: \R^d \times \R^s \to \C$ be a locally integrable function and let $\S$ be a collection of dyadic cubes in $\R^d$. We define 

\begin{multline*}
   \| b\|^2_{rect, \S}  \\
       := \sup_{ Q \in \S, R \subset \R^s \text{ cube}}    \frac{1}{|Q| |R|} \int_Q \int_R | b(x,y) - \langle b \rangle_Q(y) - \langle b \rangle_R(s) + \langle b \rangle_{Q \times R}|^2 dy dx.
\end{multline*}
\noindent
In case that $\S$ is the collection of all dyadic cubes in $\R^d$, we want to write $\| b\|^2_{rect, \Dy}$. Note that this is not quite the usual dyadic rectangular $BMO$ norm (see e.g. \cite{BlPo05}), but a mixture of a 
dyadic and a non-dyadic rectangular BMO norm.

\begin{cor}   \thlabel{cor:itcom}
Let $b: \R^d \times \R^s \rightarrow \C$ be a locally integrable function, let $T^{(1)}_\S$ be a sparse operator for a sparse family $\S$ in $\R^d$ as in \eqref{spOps}, and $ T^{(2)}$ be
a Calder{\'o}n-Zygmund operator on $\R^s$. Suppose that $\| b \|_{rect, \S} < \infty$. Then the iterated commutator  $ [T^{(1)}_\S, [ T^{(2)}, b ]]$, given by

\[
   [T^{(1)}_\S, [ T^{(2)}, b ]] f =  T^{(1)}_\S T^{(2)} b  -  T^{(1)}_\S b  T^{(2)} -  T^{(2)} b T^{(1)}_\S f +   b T^{(1)}_\S T^{(2)} f
\]
for suitable functions $f$ on $\R^d \times \R^s$, extends to a bounded linear operator on $L^2( \R^d \times \R^s)$.
\end{cor}
We start with the proof of \thref{cor:itcom}. 
\begin{proof}[Proof of \thref{cor:itcom}]
Let us assume for the moment that $b$ is bounded and contained in $L^2(\R^d \times \R^s)$. Let $\Hil = L^2(\R^s)$. We define 

\[
   B: \R^d \rightarrow \B(\Hil), \quad B(x) = [T^{(2)}, b(x, \cdot)].
\]
Since $b$ is bounded and $T^{(2)}$ is a Calder{\'o}n-Zygmund operator, $B(x) \in \bh$ for each $x \in \R^d$. Let $g \in L^2(\R^s)$ and $Q \in \S$, then

\begin{multline*}
   \frac{1}{|Q|} \int_Q \| B(x)g - \langle B g \rangle_Q \|^2 dx  \\
   = \frac{1}{|Q|} \int_Q \int_{\R^s} \left|  \left([T^{(2)}, b(x, \cdot)] g\right)(y) - \langle  [T^{(2)}, b(x, \cdot)] g \rangle_Q \right|^2 dy dx \\
   = \frac{1}{|Q|} \int_Q \int_{\R^s} \left|  \left([T^{(2)}, b(x, \cdot)] g\right)(x,y) - \left( [T^{(2)}, \langle b\rangle_Q] g \right)(y) \right|^2 dy dx \\
   =  \frac{1}{|Q|}  \int_{\R^s} \left\|  \left([T^{(2)}, b - \langle b\rangle_Q ] g\right)(\cdot,y) \right\|^2_{L^2(Q)} dy.\\
\end{multline*}
Note that the function $\tilde b: \R^s \rightarrow L^2(\R^d)$, defined by $\tilde b(y) = b(\cdot, y) - \langle b \rangle_Q(y)$, belongs to $BMO(L^2(Q))$ with norm less or equal to 
$|Q|^{1/2} \|b\|_{BMO_{rect, \S}}$,
since for any cube $R \subset \R^s$, we have

\begin{multline*}
       \frac{1}{|R|}\int_R \| \tilde b(y) - \langle \tilde b \rangle_R \|^2_{L^2(Q)} dy  \\= \frac{1}{|R|} \int_R \int_Q \left| b(x,y) - \langle b \rangle_Q(y) - \langle b \rangle_R(x) + \langle b \rangle_{Q \times R}     \right|^2 dx dy 
       \le |Q|  \|b\|^2_{BMO_{rect, \S}}.
\end{multline*}
Noting that the Coifman-Rochberg-Weiss Theorem for commutators \cite{coifman1976factorization} holds even in the case of Hilbert-space valued functions, we hence find that

\begin{equation} \label{ComEq}
   \frac{1}{|Q|} \int_Q \| B(x)g - \langle B g \rangle_Q \|^2 dx  
   \lesssim \frac{1}{|Q|} \| \tilde b \|_{BMO(L^2(Q))}^2 \|g\|^2_{L^2(\R^s)} \le  \|b\|_{BMO_{rect, \S}}^2.
 \end{equation}  
 Hence $B \in SBMO^\S$. Using \eqref{ComEq} and \thref{thm:com}, we find that
 
 \begin{multline*}
          \left\| [T^{(1)}_\S, [ T^{(2)}, b ]] \right\|_{L^2( \R^d \times \R^s) \to L^2( \R^d \times \R^s)} =   \left\| [T^{(1)}_\S, B] \right\|_{L^2( \R^d, L^2(\R^s)) \to L^2( \R^d, L^2(\R^s))} \\
          \lesssim \| B \|_{SBMO_\S} \lesssim   \|b\|_{BMO_{rect, \S}}.
 \end{multline*}
 The case for a general $b \in BMO_{rect, \S}$ follows by a standard approximation argument. 
\end{proof}

\begin{proof}[Proof of \thref{thm:com}] We follow a calculation from \cite{GPTV04} in a slightly more general setting.
For a function $B$ and a sparse family $\S$ as in the statement of the Theorem, we define the operator-valued weight $W_B:\R^d \to \B(\Hil \oplus \Hil)$ by
\[
        W_B(x) = V_B^*(x) V_B(x) =    \left(  \begin{matrix}   \einsH & 0 \\ 
                                     B^*(x)   &    \einsH   \end{matrix}   \right)
                                      \left(  \begin{matrix}   \einsH & B(x) \\ 
                                     0  &    \einsH   \end{matrix}   \right)
\]                 
We claim that there exists a numerical constant $c>0$, such that

\begin{equation}   \label{a2id}
  \frac{1}{c} [ W_B ]_{A_2^\S} \leq \| B \|_{SBMO^\S}^2+1   \leq c [ W_B ]_{A_2^\S}.
\end{equation}
Let $\rho(A)$ denote the spectral radius of an element $A \in \B(\Hil \oplus \Hil)$ and recall that $\rho(A) = \|A\|$ for positive operators $A$. Note that 

\[
        W_B^{-1}(x) =  V_B^{-1}(x) (V_B^*(x))^{-1}   = \left(  \begin{matrix}   \einsH & -B(x) \\ 
                                     0   &    \einsH   \end{matrix}   \right)
                                      \left(  \begin{matrix}   \einsH & 0 \\ 
                                     - B^*(x)  &    \einsH   \end{matrix}   \right).
\]
For any cube $Q \subset \R^d$, we may compute

 \begin{multline*}
    \left\|  \langle W_B \rangle_Q^{1/2}     \langle W_B^{-1} \rangle_Q^{1/2}  \right\|^2
    = \rho \left(  \langle  W_B \rangle_Q    \langle W_B^{-1} \rangle_Q    \right) = \\
     \rho\left( \frac{1}{|Q|^2} \int_Q \int_Q \left(  \begin{matrix}   \einsH & B(x) \\ 
    	B^*(x)   &    \einsH + B^*(x) B(x)  \end{matrix}   \right)
    \left(  \begin{matrix}   \einsH + B(y) B^*(y) & -B(y) \\ 
    	- B^*(y)  &    \einsH   \end{matrix} \right) dx dy \right) =\\
     \rho\left( \frac{1}{|Q|^2} \int_Q \int_Q \left(  \begin{matrix}   \einsH + B(y) B^*(y) - B(x) B^*(y) & B(x) - B(y) \\ 
    	\star   &    \einsH + B^*(x) B(x)  - B^*(x) B(y) \end{matrix}   \right) dx dy \right) \\
      = \rho\left( 
       \begin{matrix}   \einsH + \langle B B^* \rangle_Q - \langle B\rangle_Q \langle B^* \rangle_Q & 0 \\ 
    	\star   &    \einsH +  \langle B^*  B\rangle_Q  - \langle B^* \rangle_Q  \langle B \rangle_Q \end{matrix}  \right) \\
    = \max\{ \|  \einsH + \langle B B^* \rangle_Q - \langle B\rangle_Q \langle B^* \rangle_Q \|_{\bh} \, , \| \einsH +  \langle B^*  B\rangle_Q  - \langle B^* \rangle_Q  \langle B \rangle_Q \|_{\bh}   \}.
     \end{multline*}
In the previous paragraphs, we denoted the lower non-diagonal element of the matrix by $\star$, due to its lack of relevance when computing spectral radius of a lower triangular matrix. Now noting that
 $$
 \left(  \langle B B^* \rangle_Q - \langle B\rangle_Q \langle B^* \rangle_Q e, e \right)_{\Hil}  = \frac{1}{|Q|}\int_Q \| (B^* (x) - \langle B^* \rangle_Q)e \|^2 dx
 $$
 and
    $$
   \left(  \langle B^* B\rangle_Q - \langle B^* \rangle_Q \langle B\rangle_Q e, e \right)_{\Hil} = \frac{1}{|Q|}\int_Q \| (B(x) - \langle B \rangle_Q)e \|^2 dx.
   $$ 
This proves the claim in \eqref{a2id}. On the other hand, we also have
 \begin{multline*}
   \| T_\S\|_{L^2_{W_B} \to L^2_{W_B}} = \| V_B T_\S {V_B}^{-1}\|_{L^2(\Hil \oplus \Hil) \to L^2(\Hil \oplus \Hil) } \\
   =\left\| \left( \begin{matrix}   T_\S & [T_\S,B]    \\ 0 & T_\S \end{matrix} \right)\right\|_{L^2(\Hil \oplus \Hil) \to L^2(\Hil \oplus \Hil) .}
 \end{multline*}
According to \thref{SpL2W} and \eqref{a2id}, this is enough to conclude the proof of this Theorem. 
\end{proof}
%
%
\textbf{Acknowledgement.} 
We gratefully acknowledge support by the VR grant 2015-05552
\bibliographystyle{alpha}
\bibliography{Harref1}

\newcommand{\etalchar}[1]{$^{#1}$}
\begin{thebibliography}{GPTV04}

\bibitem[AC12]{AleCon12}
Alexandru Aleman and Olivia Constantin.
\newblock {The Bergman projection on vector-valued L2-spaces with
  operator-valued weights}.
\newblock {\em Journal of Functional Analysis}, 262(5):2359--2378, 2012.

\bibitem[BP{\etalchar{+}}05]{BlPo05}
{\'O}scar Blasco, Sandra Pott, et~al.
\newblock {Dyadic BMO on the bidisk}.
\newblock {\em Revista Matem{\'a}tica Iberoamericana}, 21(2):483--510, 2005.

\bibitem[CRW76]{coifman1976factorization}
Ronald~R Coifman, Richard Rochberg, and Guido Weiss.
\newblock {Factorization theorems for Hardy spaces in several variables}.
\newblock {\em Annals of Mathematics}, pages 611--635, 1976.

\bibitem[GPTV01]{GPTV01}
Thomas~Alstair Gillespie, Sandra Pott, Serguei Treil, and Alexander Volberg.
\newblock Logarithmic growth for matrix martingale transforms.
\newblock {\em Journal of the London Mathematical Society}, 64(3):624--636,
  2001.

\bibitem[GPTV04]{GPTV04}
TA~Gillespie, S~Pott, S~Treil, and A~Volberg.
\newblock {Logarithmic growth for weighted Hilbert transforms and vector Hankel
  operators}.
\newblock {\em Journal of Operator Theory}, pages 103--112, 2004.

\bibitem[Gra09]{Gra14}
Loukas Grafakos.
\newblock {Modern Fourier analysis}.
\newblock 250, 2009.

\bibitem[HP13]{HytPer13}
Tuomas Hyt{\"o}nen and Carlos P{\'e}rez.
\newblock {Sharp weighted bounds involving $A_{\infty}$}.
\newblock {\em Anal. PDE}, 6(4):777--818, 2013.

\bibitem[HPR12]{hytonen2012sharp}
Tuomas Hyt{\"o}nen, Carlos P{\'e}rez, and Ezequiel Rela.
\newblock {Sharp reverse H\"older property for $A_\infty$ weights on spaces of
  homogeneous type}.
\newblock {\em Journal of Functional Analysis}, 263(12):3883--3899, 2012.

\bibitem[HW20]{huo2020weighted}
Zhenghui Huo and Brett~D Wick.
\newblock Weighted estimates of the bergman projection with matrix weights.
\newblock {\em arXiv preprint arXiv:2012.13810}, 2020.

\bibitem[Hyt12]{Hyt12Sh}
Tuomas~P Hyt{\"o}nen.
\newblock {The sharp weighted bound for general Calder{\'o}n-Zygmund
  operators}.
\newblock {\em Annals of mathematics}, pages 1473--1506, 2012.

\bibitem[Hyt17]{Hyt17}
Tuomas Hyt{\"o}nen.
\newblock {Dyadic analysis and weights}.
\newblock {\em Lecture notes-Course University of Helsinki}, 2017.

\bibitem[IPT20]{isralowitz2020commutators}
Joshua Isralowitz, Sandra Pott, and Sergei Treil.
\newblock Commutators in the two scalar and matrix weighted setting.
\newblock {\em arXiv preprint arXiv:2001.11182}, 2020.

\bibitem[Ler13]{Ler13}
Andrei~K Lerner.
\newblock {A simple proof of the $A_{2}$ conjecture}.
\newblock {\em International Mathematics Research Notices},
  2013(14):3159--3170, 2013.

\bibitem[LN19]{LerNaz19}
Andrei~K Lerner and Fedor Nazarov.
\newblock {Intuitive dyadic calculus: the basics}.
\newblock {\em Expositiones Mathematicae}, 37(3):225--265, 2019.

\bibitem[NPTV17]{NPTV17}
Fedor Nazarov, Stefanie Petermichl, Sergei Treil, and Alexander Volberg.
\newblock {Convex body domination and weighted estimates with matrix weights}.
\newblock {\em Advances in Mathematics}, 318:279--306, 2017.

\bibitem[NTV97]{NTV97}
Fedor Nazarov, Serguei Treil, and Alexander Volberg.
\newblock {Counterexample to the infinite dimensional Carleson embedding
  theorem}.
\newblock {\em Comptes Rendus de l'Academie des Sciences-Serie I-Mathematique},
  325(4):383--388, 1997.

\bibitem[PR13]{PePo13}
Sandra Pott and Maria~Carmen Reguera.
\newblock {Sharp B{\'e}koll{\'e} estimates for the Bergman projection}.
\newblock {\em Journal of Functional Analysis}, 265(12):3233--3244, 2013.

\end{thebibliography}

\end{document}